\documentclass[11pt,twoside, a4paper]{amsart}

\usepackage{mathptmx}
\usepackage{amscd}
\usepackage{amssymb}
\usepackage{eucal}
\usepackage{amsxtra}
\usepackage{verbatim}
\usepackage{url}
\usepackage{enumerate}
\usepackage[english]{babel}
\usepackage[initials]{amsrefs}

\usepackage{comment}
\usepackage[all]{xy}
\usepackage{stmaryrd}

\DeclareFontEncoding{OT2}{}{} 
\DeclareTextFontCommand{\textcyr}{\fontencoding{OT2}
    \fontfamily{wncyr}\fontseries{m}\fontshape{n}\selectfont}

\theoremstyle{plain}

\newtheorem{theorem}{Theorem}
\newtheorem{proposition}[theorem]{Proposition}
\newtheorem{lemma}[theorem]{Lemma}
\newtheorem{corollary}[theorem]{Corollary}

\newtheorem{conditional-result}[theorem]{Conditional Result}

\newtheorem{theorem?}{Theorem(?)} [section]
\newtheorem{proposition?}[theorem]{Proposition(?)}
\newtheorem{lemma?}[theorem]{Lemma(?)}
\newtheorem{corollary?}[theorem]{Corollary(?)}

\newtheorem*{theorem*}{Theorem}
\newtheorem*{proposition*}{Proposition}
\newtheorem*{lemma*}{Lemma}
\newtheorem*{corollary*}{Corollary}
\newtheorem*{question*}{Question}
\newtheorem*{conjecture*}{Conjecture}
\newtheorem*{claim*}{Claim}

\newtheorem*{introtheorem*}{Theorem}
\newtheorem*{introproposition*}{Proposition}
\newtheorem*{introlemma*}{Lemma}
\newtheorem*{introcorollary*}{Corollary}

\theoremstyle{definition}

\newtheorem{example}[theorem]{Example}

\newtheorem*{definition*}{Definition}
\newtheorem*{example*}{Example}

\theoremstyle{remark}
\newtheorem{remark}[theorem]{Remark}

\newtheorem*{remark*}{Remark}

\DeclareSymbolFont{rsfs}{U}{rsfs}{m}{n}
\DeclareSymbolFontAlphabet{\mathcal}{rsfs}

\DeclareFontEncoding{OT2}{}{} 
\DeclareTextFontCommand{\textcyr}{\fontencoding{OT2}
    \fontfamily{wncyr}\fontseries{m}\fontshape{n}\selectfont}

\newcommand{\Sh}{\textcyr{Sh}}

\newcommand{\ZZ}{{\mathbb{Z}}}
\newcommand{\QQ}{{\mathbb{Q}}}

\newcommand{\Gal}{{\rm Gal}}
\newcommand{\Pic}{{\rm Pic}}
\newcommand{\Br}{{\rm Br}}

\newcommand{\Hom}{{\rm Hom}}
\newcommand{\Ker}{\textup{Ker}}

\newcommand{\Res}{{\rm Res}}
\newcommand{\Cor}{{\rm Cor}}

\def\G{{\mathbb{G}}}
\newcommand{\R}{{\textup{R}}}

\def\Z{{\ZZ}}
\def\Q{{\QQ}}

\def\D{{\mathbf{D}}}

\def\A{\mathbf{A}}

\begin{document}

\title{The unramified Brauer group of norm one tori}

\author{ Dasheng Wei}
\address{Academy of Mathematics and System Science, CAS, Beijing
  100190, P.R.China and Mathematisches Institut, Ludwig-Maximilians-Universit\"at M\"unchen,
  Theresienstr. 39, 80333 M\"unchen, Germany}

\email{dshwei@amss.ac.cn}

\date{\today}

%

\begin{abstract}
Let $k$ be a number field and $K/k$ Galois. We transform the construction of the unramified Brauer group of the norm one torus $\R^{1}_{K/k}(\G_m)$ into the construction of a special abelian extension over $K$. If $k=\Q$ and $K/\Q$ biquadratic, we explicitly construct the  unramified Brauer group of $\R^{1}_{K/\Q}(\G_m)$.
\end{abstract}


\bigskip


\bigskip

\maketitle

\section{Introduction} \label{sec.notation}

Let $T$ be a torus over a field $k$ of characteristic zero,
$X$ a principal homogeneous space of $T$, and $X^c$ a smooth compactification
of $X$. Since the Brauer group $\Br(X^c) :=
H^2_{\acute{e}t}(X^c,\G_m)$ is a birational invariant of smooth proper varieties, it does not depend on the choice of $X^c$ but
only depends on $X$; it is called the unramified
Brauer group of $X$. Let $\Br_0(X^c)$ be the image of the natural map $\Br(k)\rightarrow \Br(X^c)$.
Formulas for $\Br (X^c)/\Br_0(X^c)$ can be found in  \cite{CTK98}. In particular, if $K/k$ is a Galois extension and $T=\R^1_{K/k}(\G_m)$ its norm one torus, then $\Br (X^c)/\Br_0 (X^c)\cong H^3(\Gal(K/k),\Z)$. If $K/k$ is cyclic or of prime degree (not necessarily Galois),
then $\Br (X^c)=\Br_0 (X^c)$.

It is well known that the Brauer-Manin obstruction to the Hasse principle and weak approximation for rational points is the only one on $X^c$ (\cite{San81}). To compute the Brauer-Manin obstruction, one needs to construct the Brauer group. Recently, Colliot-Th\'el\`ene (\cite{CT11}) gave an explicit construction for a multi-norm torus of dimension 5. However, for general tori, it is still open, even for the norm one torus $R_{K/\Q}^1(\G_m)$, where $K/\Q$ biquadratic.

The main aim of this article is to construct the unramified Brauer group for the torus $\R^1_{K/k}(\G_m)$ when $K/k$ is Galois and $H^3(k,\Z)=0$, e.g., $H^3(k,\Z)=0$ holds for any number field $k$. In \S \ref{main theorem}, we show any element in the unramified Brauer group of $\R^1_{K/k}(\G_m)$ has a form from cup-product. Furthermore, we transform the construction of the unramified Brauer group into the construction of a special abelian extension over $K$ (see Theorem \ref{br}). Some applications and examples are also given in this section. In \S \ref{biquadratic}, using results of double coverings of $\Q^{ab}/\Q$ in \cite{And,Das,YZ}, we give the explicit construction of the unramified Brauer group for the torus $R^1_{K/\Q}(\G_m)$, where $K/\Q$ biquadratic.

\section{The Brauer group of $\R^{1}_{K/k}(\G_m)$ when $K/k$ is Galois }\label{main theorem}

Let $k$ be a field with $char(k)=0$. Let $K/k$ be a finite extension of degree $m$ and $T=\R^{1}_{K/k}(\G_m)$
its norm one torus. Suppose $K=k(\omega)$, then $\{1,\omega,\cdots,\omega^{m-1}\}$ is a basis of $K$ as a $k$-vector space. Let $\Xi=x_0+x_1\omega+\cdots+x_{m-1}\omega^{m-1}$, where $x_0,\cdots,x_{m-1}$ are variables. Let $X\subset \A_k^{m}$ be the affine $k$-variety defined by $N_{K/k}(\Xi)=\alpha\in k^\times$ ($\alpha$ is fixed), which is a principal homogeneous space of $T$. Then $\bar k[X]^\times/\bar k^\times\cong \widehat T$ as $\Gal(\bar k/k)$-modules, where $\bar k[X]^\times$ is the group of invertible regular functions of $\overline{X}$ and $\widehat T$ is the character group of $T$.

Let $X^c$ be a smooth compactification of $X$. Let $\Br(X)$ (resp. $\Br(X^c)$) be the Brauer group of $X$ (resp. $X^c$) and $\Br_0(X)$ (resp. $\Br_0(X^c)$) the image of $\Br(k)$ in $\Br(X)$ (resp. $\Br_0(X^c)$).

\begin{lemma}\label{lemma:contain} $\Br(X^c)$ is contained in the image of $H^2(k,\bar k[X]^\times)$ in $\Br(X)$.
\end{lemma}
\begin{proof}
By the Hochschild-Serre Spectral sequence, we have
$$\aligned \Pic(\overline X)^{\Gal(\bar k/k)} \rightarrow H^2(k,\bar k[X]^\times)
\rightarrow \Ker[\Br(X)\rightarrow \Br(\overline{X})]\rightarrow H^1(k,\Pic(\overline X)).
\endaligned$$
Since $\overline X\cong \G_m^d$ over $\bar k$, we have $\Pic(\overline X)=0$, where $d=[K:k]-1$. Therefore
\begin{equation} \label{f-1} H^2(k,\bar k[X]^\times)\cong \Ker[\Br(X)\rightarrow \Br(\overline{X})].\end{equation}
Since $X^c$ is geometrically rational, we have $\Br(\overline{X^c})=0$. Therefore
$$\Br(X^c)\subset \Ker[\Br(X)\rightarrow \Br(\overline{X})].\qedhere $$
\end{proof}

Let $\Gamma_k=\Gal(\bar k/k)$ and $\Gamma_K=\Gal(\bar k/K)$. Let $\Z[K/k]$ denote the group ring $\Z[\Gamma_k/\Gamma_K]$. For $i\geq 0$, we have the cup product $$(\cdot\ , \cdot\ ):\Z[K/k]\times H^i(K,\Z) \rightarrow H^i(K,\Z[K/k]).$$
Let $\Cor_{K/k}$ be the corestriction map $H^i(K,\cdot\ )\rightarrow H^i(k,\cdot\ ).$
\begin{lemma}\label{sha} Let $\Gamma_K\in \Gamma_k/\Gamma_K$. Then $\Cor_{K/k}(\Gamma_K, \cdot\ ): H^i(K,\Z)\rightarrow H^i(k,\Z[K/k])$ is the inverse map of Shapiro's isomorphism $sh: H^i(k,\Z[K/k]) \rightarrow H^i(K,\Z)$.
\end{lemma}
\begin{proof} The case $i=0$ is obvious.
Let $i>0$. Let $g=\Cor_{K/k}(\Gamma_K, \cdot\ )$. Since $sh$ is an isomorphism (see \cite[Proposition I.1.6.3]{NSW}), it is enough to show $sh\circ g=id$. Let $C^\bullet(\Gamma_K,\Z)$ (resp. $C^\bullet(\Gamma_K,\Z[K/k])$) be the homogeneous cochain complex of $\Gamma_K$ with coefficients in $\Z$ (resp. $\Z[K/k]$) (see \cite[Chapter~1]{NSW}). Suppose $x\in H^i(K,\Z)$, choose an $i$-cocycle $u\in C^i (\Gamma_K,\Z)$ which represents $x$.

We use notations in pp. 46 of \cite[Chaper I, Section 5]{NSW}, let $c=\Gamma_K\sigma \in \Gamma_K \setminus \Gamma_k$, we choose a fix representative $\overline c\in c$ and choose $\overline c$ to be the identity if $c= \Gamma_K\in \Gamma_K\setminus \Gamma_k$. For any $(\sigma_0,\dots,\sigma_i)\in \oplus_{i+1}\Gamma_K$, then $$\aligned g(u)(\sigma_0,\dots, \sigma_i)&=\Cor_{K/k}(\Gamma_K, u)(\sigma_0,\dots, \sigma_i)\\
&=\sum_{c\in \Gamma_K\setminus \Gamma_k}u(\bar c \sigma_0 \overline{c\sigma}_0^{-1},\dots,\overline c \sigma_i \overline{c\sigma}_i^{-1})\overline c^{-1} \Gamma_K.\endaligned$$
Let $j_{\Gamma_K}$ be the projection $\Z[K/k]\rightarrow \Z$ by $\Gamma_K  \mapsto 1, \gamma \Gamma_K \mapsto 0 \text{ if } \gamma\not \in \Gamma_K$.
Since $\sigma_i \in \Gamma_K$ for any $i$, we have $\Gamma_K=\Gamma_K \sigma_i \in \Gamma_K\setminus \Gamma_k$, hence $\overline{c\sigma}_i$ is the identity for $c=\Gamma_K\in \Gamma_K \setminus \Gamma_k$. Therefore
$$\aligned (sh\circ g(u))(\sigma_0,\dots,\sigma_i)&=j_{\Gamma_K}\left(\sum_{c\in \Gamma_K\setminus \Gamma_k}u(\bar c \sigma_0 \overline{c\sigma}_0^{-1},\dots,\overline c \sigma_i \overline{c\sigma}_i^{-1})\overline c^{-1} \Gamma_K\right)\\
&=u(\sigma_0,\dots,\sigma_i). \qedhere
\endaligned$$
\end{proof}

Let $\Xi$ be as in the beginning of this section, we can see $\Xi\in \bar k[X]^\times$. For any character $\chi$ of $\Gamma_K$, it is a homomorphism from $\Gamma_K$ to $\Q/\Z$, which gives rise to an element of $H^1(K,\Q/\Z)$ and then $H^2(K,\Z)$. Then the cup-product $(\Xi,\chi)$ gives an element of $H^2(K,\bar k[X]^\times)$. Let $\Cor_{K/k}: H^2(K,\bar k[X]^\times)\to H^2(k,\bar k[X]^\times)$ be the corestriction.  Then $(\Xi,\chi)$ gives the element $\Cor_{K/k}(\Xi,\chi)\in H^2(k,\bar k[X]^\times)\subset \Br(X)$. The following Proposition will show any element of $\Br(X^c)/\Br_0(X^c)$ has this form.
\begin{proposition}\label{form} Let $k$ be a field of characteristic $0$ and let $\chi$ be a character of $\Gamma_K$.
\begin{enumerate}[(a)]
\item The element $\Cor_{K/k}(\Xi,\chi)=0\in \Br(X)/\Br_0(X)$ if and only if $\chi$ is the restriction of a character of $\Gamma_k$.
\item Suppose $H^3(k,\Z)=0$, e.g., $k$ is a number field or $p$-adic number field (see \cite[Corollary 4.7]{Milne86}). Each element of $\Br(X)/\Br_0(X)$ in the image of $H^2(k,\bar k[X]^\times)$ is of the form $\Cor_{K/k}(\Xi,\chi)$.
\end{enumerate}
\end{proposition}

\begin{proof}
Using the natural exact sequence $$0\rightarrow  \Z\rightarrow
\Z[K/k]\rightarrow \widehat{T}\rightarrow 0,$$ we deduce the following exact sequence
\begin{equation}\label{cor}
H^2(k,\Z)\rightarrow H^2(k,\Z[K/k])\buildrel f \over \rightarrow
H^2(k,\widehat{T})\rightarrow H^3(k,\Z).
\end{equation}

Define the $\Gamma_k$-morphism $j:\Z[K/k]\rightarrow \bar k[X]^\times$ by $\Gamma_K\mapsto \Xi$.
Then we have the maps $\Z[K/k]\rightarrow \bar k[X]^\times\rightarrow \bar k[X]^\times /\bar k^\times \ (\cong \widehat T)$.
Therefore we have the induced map $H^2(k,\Z[K/k])\to H^2(k,k[X]^\times) \xrightarrow{g} H^2(k,\widehat T)$, and the composite map $H^2(k,\Z[K/k])\rightarrow H^2(k,\widehat T)$ is coincide with the map $f$ in (\ref{cor}).

By the basic exact sequence $0\rightarrow \bar k\rightarrow \bar k[X]^\times\rightarrow \bar k[X]^\times/\bar k^\times (\cong \widehat T)\rightarrow 0$, we have the following exact sequence
\begin{equation}\label{br-basic}
\Br(k) \rightarrow H^2(k,\bar k[X]^\times) \xrightarrow{g} H^2(k,\widehat{T}).
\end{equation}
Hence $g$ induces an injection $g': H^2(k,\bar k[X]^\times)/\text{Ima}[\Br(k)] \hookrightarrow H^2(k,\widehat{T})$, where $\text{Ima}[\Br(k)]$ is the image of $\Br(k)$ in $H^2(k,\bar k[X]^\times)$. We have the following commutative diagram
\begin{equation}\label{e-1-1} \xymatrix @R=12 pt @C=12pt @M=4pt{
H^2(k,\Z[K/k]) \ar[r] \ar[dr]_{f}  &  H^2(k,\bar k[X]^\times)/\text{Ima}(\Br(k))\ar[d]^{g'}\\
& H^2(k,\widehat{T}).}
\end{equation}
Since $g'$ is injective, by diagram (\ref{e-1-1}), the exact sequence (\ref{cor}) gives the following exact sequence
\begin{equation}\label{cor-1}
H^2(k,\Z)\rightarrow H^2(k,\Z[K/k])\rightarrow H^2(k,\bar k[X]^\times)/\text{Ima}(\Br(k)).
\end{equation}

By Lemma \ref{sha} and functoriality of cup product, we have the following commutative diagram
$$
\xymatrix @R=12 pt @C=12pt @M=4pt{
&H^2(K,\Z)\ar[d]^{g_1}\ar[dr]^{g_2}\\
H^2(k,\Z)\ar[ur]^{\Res_{k/K}} \ar[r]   & H^2(k,\Z[K/k])  \ar[r]  &H^2(k,\bar k[X]^\times),}
$$
where $g_1=\Cor_{K/k}(\Gamma_K,\cdot)$ is  the inverse of Shapiro's isomorphism (by Lemma \ref{sha}), $g_2=\Cor_{K/k}(\Xi,\cdot)$ and the lower right map is induced by $j$ with $j(\Gamma_K)=\Xi$. By this commutative diagram, (\ref{cor-1}) is identity with the following exact sequence
\begin{equation}\label{fun}
H^2(k,\Z)\buildrel \Res_{k/K} \over \rightarrow H^2(K,\Z)\buildrel  h \over \rightarrow
H^2(k,\bar k[X]^\times)/\text{Ima}[\Br(k)],
\end{equation}
where $h=\rho \circ \Cor_{K/k}(\Xi,\cdot)$ and $\rho$ is the quotient map $$H^2(k,\bar k[X]^\times)\rightarrow H^2(k,\bar k[X]^\times)/\text{Ima}[\Br(k)].$$

Since the map $H^2(k,\bar k[X]^\times)\rightarrow \Br(X)$ is injective by  (\ref{f-1}), the map $$H^2(k,\bar k[X]^\times)/\text{Ima}[\Br(k)]\rightarrow \Br(X)/\Br_0(X)$$ is also injective.
By the exact sequence $(\ref{fun})$, we immediately have $\Cor_{K/k}(\Xi,\chi)$ is zero in $\Br(X)/\Br_0(X)$ if and only if $\chi$ is the restriction of a character of $\Gamma_k$. Then we prove part (a).

Now we prove part (b). Since $H^3(k,\Z)=0$, the map $f$ in sequence (\ref{cor}) is surjective.
The following  diagram is commutative
\begin{equation*} \xymatrix @R=12 pt @C=12pt @M=4pt{
H^2(k,\Z[K/k]) \ar[r] \ar[dr]_{f}  &  H^2(k,\bar k[X]^\times)\ar[d]^{g}\\
& H^2(k,\widehat{T}).}
\end{equation*}
Then $g$ is also surjective since $f$ is surjective. The sequence (\ref{br-basic}) induces  $f': H^2(k,\bar k[X]^\times)/\text{Ima}[\Br(k)] \rightarrow H^2(k,\widehat{T})$ is an isomorphism.  By the sequence (\ref{cor}) and diagram (\ref{e-1-1}), the second map in (\ref{cor-1}) is also surjective, hence the map $h$ in (\ref{fun}) is also surjective. Therefore  each element of $H^2(k,\bar k[X]^\times)/\text{Ima}[\Br(k)]$ is of the form $\Cor_{K/k}(\Xi,\chi)$ by (\ref{fun}), where $\chi$ is a character of $\Gamma_K$. Therefore each element of $\Br(X)/\Br_0(X)$ in the image of $H^2(k,\bar k[X]^\times)$ is of the form $\Cor_{K/k}(\Xi,\chi)$.
\end{proof}

In the following we will give the description of the unramified Brauer group of $X$. First we give some notation.

Suppose $K/k$ is Galois. {\it We say a field extension $L/K$ satisfies the condition $(*)$ over $k$ if:
$$L/K'\text{ is abelian for any subfield }
K'\subset K \text{ containing } k \text{ and satisfying } K/K' \text{ is cyclic.}$$}
Obviously $L/K$ is abelian if $L/K$ satisfies the condition $(*)$ over $k$.

\begin{lemma} \label{cen} Suppose $K/k$ is Galois and $L/K$ satisfies the condition $(*)$ over $k$. Then $L/k$ is Galois and $\Gal(L/K)$ is contained in the center of $\Gal(L/k)$.
\end{lemma}
\begin{proof} We will show
$\sigma(L)=L$ for any $\sigma \in \Gamma_k$. Let $K^{\sigma}$
be the fixed subfield of $\sigma$ in $K$. Since $K/K^\sigma$ is
cyclic, we have $L/K^{\sigma}$ is abelian by the condition $(*)$ over $k$. Since $\sigma \in \Gamma_{K^\sigma}$, it implies
$\sigma(L)=L$. Therefore $L/k$ is Galois.

Let $\sigma\in \Gal(L/K)$. For any $ g\in \Gal(L/k)$, let $K^g$ be the subfield of $K$ fixed by $g$. Hence $K/K^g$ is cyclic. By the condition $(*)$ over $k$, $\Gal(L/K^g)$ is abelian. This implies $\sigma g=g\sigma$ since $\sigma, g\in \Gal(L/K^g)$, hence $\sigma$ is contained in the center of $\Gal(L/k)$.
\end{proof}

\begin{theorem} \label{br} Let $k$ be a field of characteristic $0$. Suppose $K/k$ is Galois. Let $\chi$ be a character of $\Gamma_K$. Then:
\begin{enumerate}[(a)]

\item $\Cor_{K/k}(\Xi,\chi)\in \Br(X^c)$ if and only if $\chi$ can factor through an abelian extension $L/K$ which satisfies the condition $(*)$ over $k$.

\item If $H^3(k,\Z)=0$, then any element of $\Br(X^c)/\Br_0(X^c)$ is of the form $\Cor_{K/k}(\Xi,
\chi)$.

\end{enumerate}
\end{theorem}
\begin{proof}
Part (b) follows from Lemma \ref{lemma:contain} and Proposition \ref{form}. So it remains to prove part (a).

Let $k(X)$ be the function field of $X$. Let $A$ be a discrete
valuation ring containing $k$ with fraction field $k(X)$ and residue field $\kappa_A$. There
is a residue map $$\partial_A: \Br(k(X))\rightarrow
H^1(\kappa_A,\Q/\Z).$$ By Grothendieck's purity theorem, we have
$$\Br(X^c)=\bigcap_{A}\Ker(\partial_A)\subset \Br(k(X)),$$
where $A$ runs through all discrete valuation rings as above.

Suppose $L/K$ satisfies the condition $(*)$ over $k$. Let $\chi$ be a character of $\Gamma_K$ which factors through $\Gal(L/K)$. Let $B=\Cor_{K/k}(\Xi, \chi)\in \Br(X)$. We want to prove
$B\in \Br(X^c)$. Therefore we only need to prove the triviality of
$\partial_A(B)$ for any such discrete valuation ring $A$, $i.e.$, $\partial_A(B)(g)=0$ for any $g \in
\Gal(\bar \kappa_A/\kappa_A)$. 

We extend the embedding $k \subset
\kappa_A$ to an embedding $\bar k \subset \bar \kappa_A$, so that $g$
acts also on $\bar k$. Let $K^g$ be the subfield of $K$ fixed
by $g$, with cyclic Galois group $\Gal(K/K^g)$.

Since $X$ is geometrically integral, the function field $K(X)
\cong k(X) \otimes_k K$ of $X\times_k K$ is finite over $k(X)$. Since $k\subset A$ and $K/k$ is finite (and etale), $A\otimes_k K$ is finite etale over $A$. We can extend $A$ to a discrete valuation
ring $A_{K}$ of $K(X)=\text{Frac}(A\otimes_k K)$ with residue field
$\kappa_{A_{K}}=\kappa_A . K$. Indeed, the completion of
$k(X)$ for the given valuation is isomorphic to $\kappa_A((\pi))$,
where $\pi$ is a uniformizer. Considering the valuation given by $\pi$
on $(K . \kappa_A)((\pi))$ and using $K(X) = k(X) \otimes_k
K$, we see that the completion $\widehat{K(X)}$ of $K(X)$ with
respect to $A_{K}$ is $\kappa_{A_{K}}((\pi))$. This defines a
discrete valuation on $K(X)$ with valuation ring $A_{K}$.

For any intermediate field $M$ of $K/k$, we have similarly the
valuation ring $A_M$ of $M(X)$ with residue field $\kappa_{A_M}$. We
write $A_g$ for $A_{K^{g}}$.

By \cite[Proposition~1.1.1]{CTSD94}, we have the commutative
diagram
\begin{equation*}
    \begin{CD}
      \Br(k(X)) @>\partial_A>> H^1(\kappa_A,\Q/\Z)\\
      @V\Res_{k/K^{g}}VV @VV\Res_{\kappa_A/\kappa_{A_g}}V\\
      \Br(K^{g}(X)) @>\partial_{A_g}>> H^1(\kappa_{A_g},\Q/\Z).
    \end{CD}
\end{equation*}
Since $\kappa_{A_g} = \kappa_A . K^{g}$, we have $g \in
\Gal(\bar \kappa_A/\kappa_{A_g})$. Hence
\begin{equation}\label{eq:g}
    \partial_A(B)(g) = \Res_{\kappa_A/\kappa_{A_g}}(\partial_A(B))(g)
    = \partial_{A_g}(\Res_{k/K^{g}}(B))(g).
\end{equation}

Let $$G=\Gamma_k, U=\Gamma_{K^g} \text{ and } H=\Gamma_K.$$ Choose a system of representatives $\sigma$ of the double coset decomposition
\begin{equation} \label{coset} G=\bigsqcup_{\sigma}U\sigma H=\bigsqcup_{\sigma}U\sigma,
\end{equation}
since $H$ is normal in $G$.
By \cite[Proposition I.1.5.6]{NSW}, we have
$$\Res_{G/U}\circ \Cor_{H/G}=\sum_\sigma \Cor_{U\cap \sigma H\sigma^{-1}/U}\circ \sigma_*\circ \Res_{H/H\cap \sigma U\sigma^{-1}},$$
where $\sigma$ runs through all elements in (\ref{coset}). Note that $H$ is a normal subgroup of $G$ and contained in $U$. Therefore, we have
$$\Res_{G/U}\circ \Cor_{H/G}=\sum_\sigma \Cor_{H/U}\circ \sigma_*.$$
Therefore
$$\begin{aligned}
\Res_{k/K^{g}}(B)&=\Res_{G/U}\circ \Cor_{H/G}(\Xi,\chi)=\sum_\sigma \Cor_{H/U}\circ \sigma_*(\Xi,\chi)\\
&=\sum_\sigma \Cor_{H/U}(\sigma(\Xi),\chi^\sigma),
\end{aligned}$$
where $\chi^\sigma$ is the conjugate character of $\chi$ by $\sigma$, $i.e.$, $\chi^\sigma(g):=\chi(\sigma^{-1}g\sigma)\in \Q/\Z$ for any $g\in \Gamma_K$.

Since $\Gal(L/K)$ is contained in the center of $\Gal(L/k)$ by Lemma \ref{cen}, we have $$\chi^\sigma =\chi.$$
Hence
$$\Res_{k/K^{g}}(B)=\sum_\sigma \Cor_{H/U}(\sigma(\Xi),\chi).$$
Note that $K/K^g$ is cyclic, we have $\Gal(L/K^g)$ is abelian by the condition $(*)$ over $k$. Therefore, we can choose a character $\hat \chi$ of $\Gal(\bar k/K^g)$ which factors through $\Gal(L/K^g)$ and lifts $\chi$. Since $\sigma$ runs through a system of representatives of $U\backslash G$, we have
$$\aligned
\Res_{k/K^{g}}(B)&=\sum_\sigma \Cor_{H/U}(\sigma(\Xi),\Res_{U/H}(\hat\chi))=\sum_\sigma(N_{K/K^g}(\sigma(\Xi)),\hat\chi)\\
&=(N_{K/k}(\Xi),\hat\chi)=(\alpha,\hat\chi).
\endaligned$$
Obviously $$\partial_{A_g}(\Res_{k/K^{g}}(B))=v_{A_g}(\alpha)\hat \chi=0.$$
By (\ref{eq:g}), then $\partial_{A}(B)(g)=0$ for any $g\in \Gal(\bar \kappa_A/\kappa_A)$.
Therefore $B\in \Br(X^c)$.

On the other hand, suppose $\chi$ does not factor through any abelian extension over $K$ which satisfies the condition $(*)$ over $k$. In the following we will show that $B=\Cor_{K/k}(\Xi, \chi)\in \Br(X)$ does not lies in $\Br(X^c)$.

Let $L/K$ be the minimal abelian (cyclic) extension which $\chi$ factors through, which is the fixed field by the kernel of $\chi: \Gamma_K\to \Q/\Z$.  Then there exists a subfield $F\subset K$ with
$\Gal(K/F)$ cyclic, such that $L/F$ is not abelian. Let
 $f: \Br(X)\rightarrow \Br(X_F)$ be the natural map. Let $U'=\Gamma_{F}$. Choose a system of representatives $\sigma'$ of the double coset decomposition
\begin{equation} \label{coset2} G=\bigsqcup_{\sigma}U'\sigma' H.
\end{equation}
By \cite[Proposition I.1.5.6]{NSW}, similarly as above we have $$f(B)=\Res_{G/U'}\circ \Cor_{H/G}(\Xi,\chi)=\sum_{\sigma'} \Cor_{H/U'}(\sigma'(\Xi),\chi^{\sigma'}).$$
It's clear that $X_F$ is the affine variety
defined by $$\prod_{\sigma'} N_{K/F}(\sigma'(\Xi))=\alpha,$$ where $\sigma'$ runs through all elements in the equation (\ref{coset2}) and all $\sigma'(\Xi)$ are independent $K$-variables. Let $X'$ be the affine variety over $F$ defined by $N_{K/F}(\Xi')=\alpha$.
There is a closed immersion $\varphi: X' \rightarrow X_F$ defined by
$$\Xi' \mapsfrom \Xi,\text{ otherwise } 1 \mapsfrom \sigma'(\Xi) \text{ if } \sigma' \not \in \Gamma_F.$$
Hence
$$\varphi^*f(B)=\sum_\sigma \varphi^*(\Cor_{H/U'}(\sigma'(\Xi),\chi^{\sigma'}))=\Cor_{H/U'}(\Xi',\chi),$$
where $\varphi^*: \Br(X_F)\rightarrow \Br(X')$ is induced by $\varphi$.

We claim that $\chi$ is not the restriction of a character of $\Gamma_F$. Otherwise, suppose $\chi$ is the restriction of $\chi': \Gamma_F\to \Q/\Z$. Let $L'$ be the cyclic field fixed by $\Ker(\chi')$. Since $\Ker(\chi)=\Ker(\chi')\cap \Gamma_K$, we have $L=L'.K$. Since $L'$ and $K$ are abelian (cyclic) over $F$, then $L=L'.K$ is also abelian over $F$, which is impossible since $L/F$ is not abelian.
Therefore, by Proposition \ref{form}, we have $\varphi^*f(B)\neq 0 \in
\Br(X')/\Br_0(X')$. On the other hand, since $K/F$ is cyclic, it implies
$\Br(X'^c)=\Br_0(X'^c)$. Therefore $\varphi^* f(B)\not
\in \Br(X'^c)$, hence $B\not \in \Br(X^c)$.
\end{proof}

\begin{remark} Let $K_1,\cdots, K_m$ be finite field extensions of $k$ and $K=\cap_{i=1}^m K_i$. Let $Y$ be the variety over $k$ defined by $N_{K_1/k}(\Xi_1)\cdots N_{K_m/k}(\Xi_m)=\alpha$, which is a principle homogeneous space of the multinorm one torus associated with $K_1,\cdots, K_m$. Let $X$ be the variety over $k$ defined by $N_{K/k}(\Xi)=\alpha$. With some special assumptions, there is a canonical and explicit isomorphism $$\Br(X^c)/\Br_0(X^c)\ \xrightarrow{\simeq} \Br(Y^c)/\Br_0(Y^c)$$
by \cite[Theorem 6]{DW}. Therefore we can also construct $\Br(Y^c)$ using Theorem \ref{br}.
\end{remark}

Let $k$ be a field with characteristic $0$. Let $K=k(\sqrt{-1},\sqrt{d})$ be a biquadratic extension of $k$ . Let
$L=K(\sqrt[4]{d})$. It is clear that $L/k$ is Galois (non-abelian) and of degree $8$ with $\Gal(L/k)\cong \D_4$. Let $\chi$ be the unique nontrivial character of $\Gamma_K$ which factors through $\Gal(L/K)$.

\begin{corollary}\label{t-1} Let $k$ be a field with characteristic $0$. Let $K=k(\sqrt{-1},\sqrt{d})$ be a biquadratic extension of $k$. Then $\Cor_{K/k}(\Xi, \chi)$ is the unique
generator of $\Br(X^c)/\Br_0(X^c)$.
\end{corollary}
\begin{proof} Since $char(k)=0$, we have $\Br(X^c)/\Br_0(X^c)\cong \Sh_\omega^2(\widehat T)$ by \cite[Proposition 9.5]{CT/S87-1} or \cite{CTK98}. Since $K/k$ is Galois, we have $\Sh_\omega^2(\widehat T)\cong H^3(\Gal(K/k),\Z)$ by a similar argument as in \cite[Theorem 6.11]{PR94}. Therefore we have $$\Br(X^c)/\Br_0(X^c)\cong H^3(\Gal(K/k),\Z)\cong \Z/2\Z$$ by a classical computation of Schur (see for instance \cite[Corollary 2.2.12]{Ka}). 

By Theorem \ref{br} and the fact
that each group of order $4$ is abelian, we have $\Cor_{K/k}(\Xi, \chi)\in \Br(X^c)$. 
So it remains to show that $\Cor_{K/k}(\Xi, \chi)$ is nontrivial in $\Br(X^c)$. By Proposition \ref{form}, we only need to show $\chi$ is not the restriction of a character of $\Gamma_k$. Otherwise, suppose $\chi$  is the restriction of $\chi': \Gamma_k\to \Q/\Z$. Let $L'$ be the cyclic field fixed by $\Ker(\chi')$. Since $\Ker(\chi)=\Ker(\chi')\cap \Gamma_K$, we have $L=L'.K$. Since $L'$ and $K$ are abelian over $k$, it implies $L$ is also abelian over $k$, which is impossible since $\Gal(L/k)=\D_4$ is not abelian.
\end{proof}

We will use Corollary \ref{t-1} to give an explicit
example.

\begin{example} \label{exa:biquadratic-1}
For any positive rational number $\alpha$, we can write $\alpha= 2^{s_1}17^{s_2}\cdot p_1^{e_1}\cdots p_g^{e_g}$, where $ p_1,\cdots, p_g$ are distinct primes which do not divide $34$, and $s_1$ and $s_2$ are integers, and $e_1, \cdots, e_g$ are nonzero integers. Let $\alpha_1=p_1^{e_1}\cdots p_g^{e_g}$. Let $D(\alpha)=\{p_1, \cdots, p_g \}$.
Denote
$$\aligned
& D_1=\{p\in D(\alpha): \left(\frac{17}{p}\right)=\left(\frac{-17}{p}\right)=-1 \}\cr
& D_2=\{p\in D(\alpha): \left(\frac{17}{p}\right)=\left(\frac{-1}{p}\right)=1 \text{ and } \left(\frac{17}{p}\right)_4=-1\},
\endaligned $$
where $\left(\frac{\ \cdot\ }{\ \cdot\ }\right)$ is the Legendre symbol and $\left(\frac{\ \cdot\ }{\ \cdot\ }\right)_4$ is the quartic residue symbol.
Let $K=\Q(\sqrt{-1},\sqrt{17})$. Then the equation
\begin{equation}\label{equ:biq-1}
N_{K/\Q}(\Xi)=\alpha
\end{equation}
is solvable over $\Q$ if and only if the following conditions hold:
\begin{enumerate}[(i)]
\item (locally solvable):
$\left(\frac{\alpha_1,-1}{2}\right)=\left(\frac{\alpha_1}{17}\right)=1$, where $\left(\frac{\cdot\ ,\ \cdot}{2}\right)$ is the Hilbert symbol over $\Q_2$; $e_i$ is even if  $\left(\frac{-1}{p_i}\right)=-1$ or $\left(\frac{17}{p_i}\right)=-1$.

\item (BM-obstruction): $(-1)^{s_1+\sum_{p_i\in D_1}e_i/2+\sum_{p_i\in D_2}e_i}=\left(\frac{\alpha_1}{17}\right)_4.$
\end{enumerate}
\end{example}
\begin{proof} First we consider the local solvability of (\ref{equ:biq-1}). It is clear that (\ref{equ:biq-1}) is solvable at the infinite place and any finite place $p\nmid 34\alpha$, and (\ref{equ:biq-1}) is solvable at place $2$ if and only if $\left(\frac{\alpha_1,-1}{2}\right)=1$, and (\ref{equ:biq-1}) is solvable at place $17$ if and only if $\left(\frac{\alpha_1}{17}\right)=1$. If $p=p_i$ with $\left(\frac{-1}{p_i}\right)=\left(\frac{17}{p_i}\right)=1$, then $K/\Q$ is totally split at $p$, then (\ref{equ:biq-1}) is solvable at $p$. If $p=p_i$ with $\left(\frac{-1}{p_i}\right)=-1$ or $\left(\frac{17}{p_i}\right)=-1$, then $K\otimes_\Q \Q_p$ is a product of two unramified extensions over $\Q_p$ of degree $2$, hence  (\ref{equ:biq-1}) is solvable at place $p_i$ if and only if $e_i$ is even.

Now we consider the Brauer-Manin obstruction. We assume (\ref{equ:biq-1}) is solvable at all places. Let $L=K(\sqrt[4]{17})$. Then $B=\Cor_{K/k}(\Xi, \chi)$ is the unique generator of $\Br(X^c)/\Br_0(X^c)$ by Corollary \ref{t-1}, where $\chi$ be the unique nontrivial character of $\Gamma_K$ which factors through $\Gal(L/K)$.
Let $p$ be a prime number (or $p=\infty$) and let $x_p = (\Xi_p) \in X(\Q_p)$ with $\Xi_p \in (K \otimes \Q_p)^*$.

If $p=\infty$, then $\chi_p$ is trivial since $K$ is totally imaginary, hence $B(x_p)=0$.

Let $p$ be a finite prime, we fix an embedding $\bar \Q \hookrightarrow \bar \Q_p$ and let $K'=\Q_p\cap K$. We can see $K'$ has the following cases:
\begin{itemize}
\item if $p=2$, then $K'=\Q(\sqrt{7})$.
\item if $p=17$, then $K'=\Q(\sqrt{-1})$.
\item if $p\nmid 34$, then $K/\Q$ is unramified at $p$. Since $\Gal(K/\Q)=\Z/2\Z\times \Z/2\Z$, $K'$ contains a biquadratic fields.
\end{itemize}
Therefore, we have $2\mid [K':\Q]$ for all $p<\infty$. By a similar arguments as in the proof of Theorem \ref{br}, in the following we will show that the restriction $\Res_{\Q/K'}(B)$ of $B$ to $\Br(X\times_\Q K')$ is the constant $(\alpha, \chi')$, where $\chi'$ is a character of $\Gamma_{K'}$ which lifts $\chi$.

Using the notation in  \cite[Proposition I.1.5.6]{NSW}, let $G=\Gamma_\Q,U=\Gamma_{K'},H=\Gamma_K\subset U$, we have the following decomposition $$G=\bigsqcup_\sigma U\sigma H=\bigsqcup_\sigma U\sigma,$$
where $\sigma$ runs through a finite system of representatives of the double cosets.

Then we have
$$\begin{aligned}
\Res_{\Q/K'}(B)&=\Res_{G/U}\circ\text{Cor}_{H/G}(\Xi,\chi)
=\sum_\sigma \text{Cor}_{U/H}(\sigma(\Xi),\chi^\sigma)\\
&=\sum_\sigma \text{Cor}_{U/H}(\sigma(\Xi),\chi) \, ,
\end{aligned}$$
where the last equation holds since $\chi=\chi^\sigma$ (note that $\chi$ has order $2$).
Since $2\mid [K':\Q]$ and $\Gal(L/\Q)=\D_4$, hence $\Gal(L/K')$ is abelian. Therefore we can lift $\chi$ to be a character $\chi'$ of $\Gamma_{K'}$ factoring through $\Gal(L/K')$ and we have
\begin{equation}\label{for:exm-1}
\begin{aligned}
\Res_{\Q/K}(B)&=\sum_\sigma \text{Cor}_{V/U}(\sigma(\Xi),\Res_{U/V}(\chi'))=\sum_\sigma (N_{K/K'}(\sigma(\Xi)),\chi')\\
&=(N_{L/\Q}(\Xi),\chi')=(\alpha, \chi').
\end{aligned}
\end{equation}
Hence we have $B(x_p)=0$ for $p \neq \infty$ and $p\nmid 34\alpha$.

If $p=2$, then $K'=\Q(\sqrt{17})$. We can choose $\chi'$ to be the generator of $H^1(K'(\sqrt[4]{17})/K',\Q/\Z)\subset H^1(K',\Q/\Z)$. However $x^4=17$ is solvable in $\Q_2$. Then we have $B(x_2)=(\alpha, \chi')_2=0$.

If $p=17$, then $K'=\Q(\sqrt{-1})$. We can choose $\chi'$ to be the generator of $H^1(K'(\sqrt[4]{17})/K',\Q/\Z)\subset H^1(K',\Q/\Z)$. Then we have $B(x_{17})=(\alpha, \chi')_{17}=(-1)^{s_1}\left(\frac{\alpha_1}{17}\right)_4$.

If $p=p_i$ with $\left(\frac{17}{p_i}\right)=1$, then $K'\supset \Q(\sqrt{17})$. We can choose $\chi'$ to be the generator of $H^1(K'(\sqrt[4]{17})/K',\Q/\Z)\subset H^1(K',\Q/\Z)$, it has degree $2$. Then $$B(x_{p_i})=(\alpha, \chi')_{p_i}=\left(\frac{\alpha, \sqrt{17}}{p_i}\right)=\left(\frac{17}{p_i}\right)^{e_i}_4.$$
We will discuss it by the following cases:
\begin{itemize}
\item if $\left(\frac{-1}{p_i}\right)=-1$, then $e_i$ is even by local solvability, hence $B(x_{p_i})=\left(\frac{17}{p_i}\right)^{e_i}_4=\left(\frac{17}{p_i}\right)^{e_i/2}=1.$
\item if $\left(\frac{-1}{p_i}\right)=1$ and $\left(\frac{17}{p}\right)_4=1$, obviously $B(x_{p_i})=1$.
\item if $\left(\frac{-1}{p_i}\right)=1$ and $\left(\frac{17}{p}\right)_4=-1$, obviously $B(x_{p_i})=\left(\frac{17}{p_i}\right)^{e_i}_4=(-1)^{e_i}.$
\end{itemize}

If $p=p_i$ with $\left(\frac{17}{p_i}\right)=-1$, then $e_i$ is even by local solvability. We have $\left(\frac{-1}{p_i}\right)=1$ or $\left(\frac{-17}{p_i}\right)=1$.
We will discuss it by the following cases:
\begin{itemize}
\item if $\left(\frac{-1}{p_i}\right)=1$, then $K'=\Q(\sqrt{-1})$. We can choose $\chi'$ to be a generator of $H^1(K'(\sqrt[4]{17})/K',\Q/\Z)\subset H^1(K',\Q/\Z)$, it is of degree $4$. Hence $B(x_{p_i})=\left(\frac{17}{p_i}\right)^{e_i}_4=\left(\frac{17}{p_i}\right)^{e_i/2}=(-1)^{e_i/2}$.
\item if $\left(\frac{-17}{p_i}\right)=1$, then $K'=\Q(\sqrt{-17})$. Since $\sqrt{17}=\sqrt{-1}\cdot \sqrt{-17}=2\sqrt{-17}\cdot(1+\sqrt{-1})^2/4$,
we have $L=K'(\sqrt{-1},\sqrt{2\sqrt{-17}})$. Therefore we can choose $\chi'$ to be contained in $H^1(L/K',\Q/\Z)\subset H^1(K',\Q/\Z)$, it is of degree $1$ or $2$. Since $e_i$ is even, we have $B(x_{p_i})=(\alpha, \chi')_{p_i}=(p_i^{e_i}, \chi')_{p_i}=1.$
\end{itemize}
By the above computation, we have
$$
 B(x_{p})=\begin{cases}
 (-1)^{s_1}\left(\frac{\alpha_1}{17}\right)_4 & \text{ if } p=17\\
 (-1)^{e_i/2} & \text{ if } p=p_i\in D_1\\
 (-1)^{e_i} & \text{ if } p=p_i\in D_2\\
 1 & \text{ otherwise.}
\end{cases}
$$
The proof follows from \cite[Theorem 8.12]{San82} or \cite[Chapter 5]{Sko01}.
\end{proof}

 Let $k$ be a field with characteristic $0$. Let $K=k(\sqrt{d_1},\sqrt{d_2})$ be a biquadratic extension of $k$. Suppose $x^2-d_1y^2=d_2z^2$ has a nonzero solution $(x_0,y_0,z_0)$ in $k$.
Let $L=K(\sqrt{x_0+y_0\sqrt{d_1}})$. We can see $L/k$ is Galois (non-abelian) and of degree $8$ with $\Gal(L/k)\cong \D_4$. Let $\chi$ be the unique nontrivial character of $\Gal(\bar k/K)$ which factors through $\Gal(L/K)$. Similarly  as above we immediately have the following result:

\begin{corollary}\label{t-2}  Let $k$ be a field with characteristic $0$. Let $K=k(\sqrt{d_1},\sqrt{d_2})$ be as above. Then $\Cor_{K/k}(\Xi, \chi)$ is the unique
generator of $\Br(X^c)/\Br_0(X^c)$.
\end{corollary}

Finally we will use Corollary \ref{t-2} to give an explicit
example.
\begin{example} \label{exm:biquadratic-2}
For any nonzero rational number $\alpha$, we can write $\alpha=(-1)^{s_0}2^{s_1} 13^{s_2}17^{s_3} p_1^{e_1}\cdots p_g^{e_g}$, where $ p_1,\cdots, p_g$ are distinct primes which do not divide $442$, and $s_0,\cdots,s_2$ are integers, and $e_1, \cdots, e_g$ are nonzero integers. Let $D(\alpha)=\{p_1, \cdots, p_g \}$ and $\alpha_1=p_1^{e_1}\cdots p_g^{e_g}$.  Let $K=\Q(\sqrt{13},\sqrt{17})$.
Denote
$$\aligned
& D_1=\{p\in D(\alpha): \left(\frac{13}{p}\right)=\left(\frac{17}{p}\right)=-1 \}\cr
& D_2=\{p\in D(\alpha): \left(\frac{13}{p}\right)=\left(\frac{17}{p}\right)=1 \text{ and } \left(\frac{15+4\sqrt{13}}{p}\right)=-1\}.
\endaligned $$
Let $K=\Q(\sqrt{13},\sqrt{17})$. Then the equation
$N_{K/\Q}(\Xi)=\alpha$ is solvable over $\Q$ if and only if the following conditions hold:
\begin{enumerate}[(i)]
\item (locally solvable):
$s_1$ is even; and $\left(\frac{\alpha_1}{13}\right)=\left(\frac{\alpha_1}{17}\right)=1$; and $e_i$ is even if  $\left(\frac{13}{p_i}\right)=-1$ or $\left(\frac{17}{p_i}\right)=-1$.

\item (BM-obstruction): $(-1)^{\sum_{p_i\in D_1}e_i/2+\sum_{p_i\in D_2}e_i}= (-1)^{s_0+s_2}\cdot \left(\frac{\alpha_1,-1}{2}\right)$, where $\left(\frac{\cdot\ ,\ \cdot}{2}\right)$ is the Hilbert symbol over $\Q_2$.
\end{enumerate}
\end{example}
\begin{proof} We can see $15^2-13\cdot 4^2=17$. Let $L=K(\sqrt{15+4\sqrt{13}})$. Then $B=\Cor_{K/k}(\Xi, \chi)$ is the unique generator of $\Br(X^c)/\Br_0(X^c)$ by Corollary \ref{t-1}, where $\chi$ be the unique nontrivial character of $\Gamma_K$ which factors through $\Gal(L/K)$.
The proof follows from a similar argument as in Example \ref{exa:biquadratic-1}.
\end{proof}

\begin{remark}
\begin{enumerate}[(a)]
\item By the Poitou-Tate duality, for any torus $T$, if the Tate-Shafarevich group $\Sh^2(\widehat T)$ is trivial, then the Hasse principle holds on any principal homogeneous space  of $T$. For the two examples as above, since $\Sh^2(\widehat T)=\Z/2\Z$ is non-trivial, the Hasse principle does not hold.

\item Let $k$ be a number field and $K=k(\sqrt{a},\sqrt{b})$ a biquadratic extension over $k$. Sansuc (\cite[Proposition 6]{San82}) gave a method to determine the solvability of the equation $N_{K/k}(\Xi)=\alpha\in k^\times$. For each $\alpha\in k^\times$, his method needs to look for a rational solution $(\Xi_1,\Xi_2)\in k(\sqrt{a})^\times \times k(\sqrt{b})^\times$ of the equation $$N_{k(\sqrt{a})/k}(\Xi_1)\cdot N_{k(\sqrt{b})/k}(\Xi_2)=\alpha.$$
Generally, looking for a rational point is not easy. By our method,  the unramified Brauer groups for all $\alpha\in k^\times$ have the same form, then the solvability for any $\alpha$ can be uniformly tested, e.g., we only need to compute some Legendre or Hilbert symbols for all $\alpha$ in Examples \ref{exa:biquadratic-1} and \ref{exm:biquadratic-2}. On the other hand, our method can also work on any Galois extension (not only biquadratic), see Example \ref{exa:cyclotomic} for a cyclotomic case.
\end{enumerate}
\end{remark}

\section{The case that $k=\Q$ and $K/\Q$ is biquadratic} \label{biquadratic}
In \S\ref{r-double}, we will recall some results of double coverings of $\Q^{ab}/\Q$. In \S\ref{construction over Q}, an explicit construction for the biquadratic case will be given using Theorem \ref{br} in \S\ref{main theorem} and the double coverings in \S\ref{r-double}.
\subsection{Double coverings of $\Q^{ab}/\Q$}\label{r-double}
Suppose $K/F$ is Galois. A double covering of $K/F$ (defined in \cite{Das}) is an extension $\hat{K}/K$ of degree $\leq 2$ such that $\hat{K}/F$ is Galois. Let $\Q^{ab}$ be the maximal abelian extension of $\Q$. In the following we will describe all double coverings of $\Q^{ab}/\Q$ (see \cite{And,Das}) and of the cyclotomic extension $\Q(\xi_n)/\Q$ (see \cite{YZ}).

Let $\mathcal A$ be the free abelian group on the symbols of the form $[a] (a\in \Q)$ modulo the identifications
$$[a]=[b]\Leftrightarrow a-b\in \Z.$$
For all odd primes $p<q$, define $\bold a_{pq}\in \mathcal A$ as
$$\bold a_{pq}=\sum_{i=1}^{\frac{p-1}{2}}\left(\left[\frac{i}{p}\right]-\sum_{k=0}^{\frac{q-1}{2}}\left[\frac{i}{pq}+\frac{k}{q}\right]\right)-
\sum_{j=1}^{\frac{q-1}{2}}\left(\left[\frac{j}{q}\right]-\sum_{l=0}^{\frac{p-1}{2}}\left[\frac{j}{pq}+\frac{l}{p}\right]\right)$$
and for prime $q>2$, define $\bold a_{2q}\in \mathcal A$ as
$$\bold a_{2q}=\left( \left[\frac{1}{4}\right]-\sum_{k=0}^{\frac{q-1}{2}}\left[\frac{k}{q}+\frac{1}{4q}\right]\right)-
\sum_{j=1}^{\frac{q-1}{2}}\left(\left[\frac{j}{q}\right]+\left[\frac{j}{q}-\frac{1}{2q}\right]-\left[\frac{j}{2q}\right]-\left[\frac{j}{2q}-\frac{1}{4q}\right]\right).$$

Let $$\sin: \mathcal A\rightarrow {\Q^{ab}}^\times$$ be the unique homomorphism such that
$$\sin[a]=\begin{cases}2\sin(\pi a) (=\mid 1-e^{2\pi i a}\mid)\ &\text{ if }0<a<1 \\ 1 &\text{ if }a=0 \end{cases}(a\in \Q\cap [0,1)).$$
The composite field of all double coverings of $\Q^{ab}/\Q$ is (see \cite[main theorem]{And})
 $$\Q^{ab}\left(\{\sqrt[4]{l}\}_{l:\text{prime}}\bigcup\{\sqrt{\sin \bold a_{pq}}\}_{\substack{p,q:\text{prime}\\ p<q}}\right).$$

Now we come to the cyclotomic case. First we give some notation.
Let $n\not\equiv 2 \mod 4$, we define a subset $S_n$ of $\Z$ associated to $n$  as following:
\begin{enumerate}[i)]
\item if $2\nmid n$, set $S_n:= \{\text {odd prime factors of } n\}$;

\item if $4\mid n$ and $8\nmid n$, set $S_n:= \{-1\}\cup \{\text {odd prime factors of } n\}$;

\item if $8\mid n$, set $S_n:=\{-1,2\}\cup \{\text {odd prime factors of } n\}$.
\end{enumerate}

 If $4\mid n$, then for all $p,q\in S_n$ and $p< q$, we set
$$u_{pq}:=\begin{cases}\sqrt{q} &\text{ if } p=-1\\
\sin \bold a_{pq} &\text{ otherwise}.
\end{cases}$$

If $2\nmid n$, then for primes $p,q\in S_n$ and $p<q$, we set
\begin{equation} \label{def-u} u_{pq}:=\begin{cases} \sin \bold a_{pq} &\text{ if } p\equiv q \equiv 1 \mod 4 \\ \sqrt{p}\cdot \sin \bold a_{pq} &\text{ if } p\equiv 1, q \equiv 3 \mod 4 \\ \sqrt{q}\cdot \sin \bold a_{pq} &\text{ if } p\equiv 3,q \equiv 1 \mod 4  \\ \sqrt{pq}\cdot \sin\bold a_{pq} &\text{ if } p\equiv q \equiv 3
 \mod 4,\end{cases}
 \end{equation}

Let $K=\Q(\xi_n)$, where $\xi_n$ is a primitive root of unity. Then the composite field of all double coverings of $K/\Q$ is (see \cite[Thoerem 1]{YZ})  $$K(\{\sqrt{u_{pq}}\}_{p<q\in S_n}).K',$$
where $K'=K(\{\sqrt{-1}\}\bigcup\{\sqrt{l}\}_{_{l:\text{prime}}})$.

\subsection{Construction of the Brauer group}\label{construction over Q}
Let $K=\Q(\sqrt{d_1 d_2},\sqrt{d_1d_3})$ such that $d_1,d_2,d_3\in \Z$ are square-free and
relatively prime to each other. Without loss generality, we can assume $d_1d_2>0$. In this section, we will explicitly construct the unramified Brauer group of the affine variety $X$ over $\Q$ defined by $N_{K/k}(\Xi)=\alpha$, where $\alpha\in \Q^\times$ fixed.

Denote
\begin{equation}\label{definition-T}
\aligned &S_i=\{p \text{ rational prime}: p\mid d_i\} \text{ for }1\leq i\leq 3,\\
& R=\bigcup_{1\leq i<j\leq 3}S_{i}\times S_{j},\\
&N=\begin{cases} |d_1 d_2 d_3|\ &\text{ if } d_1d_2\equiv d_1d_3\equiv 1 \mod 4\\
4|d_1 d_2 d_3|\ &\text{ otherwise}.
\end{cases}
\endaligned
\end{equation}

Let $F=\Q(\xi_N)$. It is clear that $K$ is contained in the cyclotomic field $F$. For simplicity of the notation,
we extend the definition of $\bold a_{pq}$ and $u_{pq}$ for $p> q$ by
$$\bold a_{pq}=\bold a_{qp} \text{ and } u_{pq}=u_{qp}.$$
Let $$\aligned &\Delta=\begin{cases}\prod_{(p,q)\in R}\sin \bold a_{pq}\ &\text{ if } d_1d_3>0\\
\sqrt{d_1d_2} \prod_{(p,q)\in R}\sin \bold a_{pq}\ &\text{ if } d_1d_3<0,
\end{cases}\\
&L=F(\sqrt{\Delta}).\endaligned$$

We will see the generator of $\Br(X^c)/\Br_0(X^c)$ is constructed by a character associated to $L/K$.

\begin{lemma} \label{gal} The field extension $L/\Q$ is Galois and $L\not \subset \Q^{ab}$.
\end{lemma}
\begin{proof}
If $4\mid N$, it follows from Theorem 11 and 12 in \cite{Das}. So we only need to consider the case $4\nmid N$, $i.e.$,  $d_1d_2\equiv d_1d_3\equiv 1 \mod 4$.

Let $\Delta'=\prod_{(p,q)\in R}u_{pq}$. 
By an easy computation, we have $$\Delta'=\prod_{(p,q)\in R}\sin \bold a_{pq}\prod_{p\mid d_1d_2d_3}\sqrt{p^{e_p}},$$
where $$e_p=\begin{cases}\#\{q|d_2d_3: q\equiv 3 \mod 4\} \ &\text{ if }p\mid d_1\\
\#\{q|d_1d_3: q\equiv 3 \mod 4\} \ &\text{ if }p\mid d_2\\
\#\{q|d_1d_2: q\equiv 3 \mod 4\} \ &\text{ if }p\mid d_3.
\end{cases}$$
Since $d_1d_2>0$, it implies $e_p$ is even when $p\mid d_3$.

If $d_1d_3>0$, we also have $d_2d_3>0$. Then $e_p$ is even when $p\mid d_1 d_2$. So $\Delta'=\pm \Delta \cdot u^2$ with $u\in F^\times$. Therefore $L=F(\sqrt{\Delta})=F(\sqrt{\Delta'})\text{ or } F(\sqrt{-\Delta'})$. Then $L/\Q$ is Galois and $L\not \subset \Q^{ab}$ by \cite[Theorem 1 and Proposition 1]{YZ}.

If $d_1d_3<0$, we also have $d_2d_3<0$. Then $e_p$ is odd when $p\mid d_1 d_2$. So $\Delta'=\pm \Delta\cdot u^2$ with $u\in F^\times$. Therefore $L=F(\sqrt{\Delta})=F(\sqrt{\Delta'})\text{ or } F(\sqrt{-\Delta'})$. Then $L/\Q$ is Galois and $L\not \subset \Q^{ab}$ by \cite[Theorem 1 and Proposition 1]{YZ}.
\end{proof}

\begin{theorem} Let $X$ be the affine variety over $\Q$ defined by $N_{K/\Q}(\Xi)=\alpha\in \Q^\times$. There is a character $\chi$ of $\Gamma_K$ which factors through $\Gal(L/K)$ and  nontrivial on $\Gal(L/F)$. Such $\chi$ gives an element $\Cor_{K/\Q}(\Xi, \chi)$ in $\Br(X^c)$ which generates $\Br(X^c)/\Br_0(X^c)$.
\end{theorem}
\begin{proof}
First we will show $L/K$ satisfies the condition $(*)$ over $\Q$, hence $L/K$ is abelian, then there exists a character $\chi$ of $\Gamma_K$ which factors through $\Gal(L/K)$ and  nontrivial on $\Gal(L/F)$. By Theorem \ref{br}, we have $\Cor_{K/\Q}(\Xi, \chi) \in \Br(X^c)$.

Let $K'$ be a subfield of $K$ such that $K/K'$ is cyclic.  We want to show $\Gal(L/K')$ is abelian. The extension $\Q^{ab}(\sqrt{\Delta})/K'$ is Galois by the main theorem in \cite{And}. Since $\Gal(L/K')$ is a quotient of $\Gal(\Q^{ab}(\sqrt{\Delta})/K')$, we only need to show $\Gal(\Q^{ab}(\sqrt{\Delta})/K')$ is abelian.
Let $G^{ab}=\Gal(\Q^{ab}/\Q)$. There is the central extension
\begin{equation*}\Sigma: 0\rightarrow \Z/2\Z\ \rightarrow \Gal(\Q^{ab}(\sqrt{\Delta})/\Q) \rightarrow G^{ab}\rightarrow 0.
\end{equation*}
Let $$\vartheta: G^{ab}\rightarrow \Gal(\Q^{ab}(\sqrt{\Delta})/\Q)$$ be any set-theoretic splitting of the central extension $\Sigma$. Then $\Sigma$ gives a $2$-cocycle $$a\in Z^2(G^{ab},\Z/2\Z)$$ defined by the formula $$a_{\sigma,\tau}=\vartheta(\sigma)\vartheta(\tau)\vartheta(\sigma\tau)^{-1}\ \   (\sigma,\tau\in G^{ab}).$$
Let $\beta \in Z^2(G^{ab},\Z/2\Z)$ defined by $$\beta_{\sigma,\tau}=a_{\sigma,\tau}-a_{\tau,\sigma}\ \  (\sigma,\tau\in \Gal(\Q^{ab}/K')).$$
It's easy to check $\beta$ is a skew-symmetric (symmetric) bilinear map.

For each odd prime $p$, let $G_p\subset G^{ab}$ be the inertia subgroup at $p$, which is isomorphic to $\Z_p^\times$. Let $G_{-1}\subset G^{ab}$ be the subgroup generated by the restriction of complex conjugation to $\Q^{ab}$. Let $G_2\subset G^{ab}$ be the subgroup of the inertial subgroup at $2$ fixing $\sqrt{-1}$, which is isomorphic to $1+4\Z_2$. Let $S=\{-1\}\cup\{ p\mid p \text{ is rational prime} \}$.
By Kronecker-Weber's theorem, we have $$G^{ab}=\prod_{p\in S}G_{p}.$$
For $p\in S$ the profinite group $G_p$ is procyclic. Let $\sigma_p\in G^{ab}$ such that $\sigma_p$ projects to a topological generator of $G_p$ and projects to $1\in G_q$ for $q\neq p$.
For $p\in S$, the natural quotient map $G^{ab}\rightarrow \Gal(\Q(\sqrt{p})/\Q)\cong \Z/2\Z$ induces the map $G_p \rightarrow \Z/2\Z$ by $\sigma_p^{i_p}\mapsto \bar i_p$.

(1) Suppose $d_1d_3>0$. Then $\Delta=\prod_{(p,q)\in R}\sin \bold a_{pq}$ by our definition.
By the Log wedge Formula in \S 4.3.4  and Proposition 2.9 in \cite{And}, we have
$$\beta=\sum_{(p,q)\in R}\delta_{p,q}\in Z^2(G^{ab},\Z/2\Z),$$
 where $\delta_{p,q}=\delta_{q,p}: G^{ab}\times G^{ab}\rightarrow \Z/2\Z$ is defined by $$((\sigma_l^{i_{l,1}})_{l\in S},(\sigma_l^{i_{l,2}})_{l\in S})\mapsto \bar i_{p,1}\bar i_{q,2}+\bar i_{p,2} \bar i_{q,1}$$
with $i_{-1,1},i_{-1,2}=0 \text{ or } 1$ and $i_{l,1},i_{l,2}\in \Z_l$ for $l\neq -1$,
 and see (\ref{definition-T}) for the definition of $R$.

(2) Suppose $d_1d_3<0$. Then $\Delta=\sqrt{d_1d_2}\prod_{(p,q)\in R}\sin \bold a_{pq}$.
By the Log wedge Formula in \S 3.4, \S 4.3.4 and Proposition 2.9 in \cite{And}, we have
$$\beta=\sum_{(p,q)\in R}\delta_{p,q}+\sum_{p\mid d_1d_2}\delta_{-1,p},$$
where $\delta_{-1,p}: G^{ab}\times G^{ab}\rightarrow \Z/2\Z$ is defined by $$((\sigma_l^{i_{l,1}})_{l\in S},(\sigma_l^{i_{l,2}})_{l\in S})\mapsto \bar i_{-1,1}\bar i_{p,2}+\bar i_{-1,2}\bar i_{p,1},$$
with $i_{-1,1},i_{-1,2}=0 \text{ or } 1$ and $i_{l,1},i_{l,2}\in \Z_l$ for $l\neq -1$.

We have the central extension
\begin{equation*}\Sigma_{K'}: 0\rightarrow \Z/2\Z\ \rightarrow \Gal(\Q^{ab}(\sqrt{\Delta})/K') \rightarrow \Gal(\Q^{ab}/K')\rightarrow 0
\end{equation*} Then $\Sigma_{K'}$ gives a $2$-cocycle $a'=\Res_{\Q/K'}(a) \in Z^2(\Gal(\Q^{ab}/K'),\Z/2\Z)$.
Let $$\beta_{\sigma,\tau}'=a'_{\sigma,\tau}-a'_{\tau,\sigma}\ \  (\sigma,\tau\in \Gal(\Q^{ab}/K')).$$ We can verify (see \cite[Lemma 2.8]{And}) that $\Gal(\Q^{ab}(\sqrt{\Delta})/K')$ is abelian if and only if
\begin{equation} \label{cond-abel} \beta_{\sigma,\tau}'=0 \text { for any } \sigma,\tau\in \Gal(\Q^{ab}/K').
\end{equation}
In the following we will check $\beta'$ satisfies (\ref{cond-abel}).

(i) Suppose $d_1d_3>0$. We can assume $K'=\Q(\sqrt{d_1d_2})$. The other cases are similar since $d_1d_2,d_1d_3$ and $d_2d_3$ are all positive.
Let $$g_1=(\sigma_p^{i_{p,1}})_{p\in S},g_2=(\sigma_p^{i_{p,2}})_{p\in S}\in \Gal(\Q^{ab}/K')\subset G^{ab}.$$
Since $g_1, g_2$ fix $K'$, we have $$\sum_{p\mid d_1 d_2}\bar i_{p,j}=\sum_{p\mid d_1}\bar i_{p,j}+\sum_{p\mid d_2}\bar i_{p,j}=0 \text{ for }j=1,2.$$
Therefore $$\aligned\beta_{g_1,g_2}'=&\beta_{g_1,g_2}=\sum_{(p,q)\in R}(\bar i_{p,1}\bar i_{q,2}+\bar i_{p,2} \bar i_{q,1})\\
=&\sum_{(p,q)\in S_{1}\times S_{2}}(\bar i_{p,1}\bar i_{q,2}+\bar i_{p,2}\bar i_{q,1})+\sum_{(p,q)\in S_{1}\times S_{3}}(\bar i_{p,1}\bar i_{q,2}+\bar i_{p,2}\bar i_{q,1})\\
&+\sum_{(p,q)\in S_{2}\times S_{3}}(\bar i_{p,1}\bar i_{q,2}+\bar i_{p,2}\bar i_{q,1})\\
=&\sum_{p\mid d_1}\bar i_{p,1} \sum_{p\mid d_2}\bar i_{p,2}+\sum_{p\mid d_1}\bar i_{p,1}\sum_{p\mid d_3}\bar i_{p,2}+\sum_{p\mid d_2}\bar i_{p,1}\sum_{p\mid d_3}\bar i_{p,2}\\
&+\sum_{p\mid d_1}\bar i_{p,2} \sum_{p\mid d_2}\bar i_{p,1}+\sum_{p\mid d_1}\bar i_{p,2}\sum_{p\mid d_3}\bar i_{p,1}+\sum_{p\mid d_2}\bar i_{p,2}\sum_{p\mid d_3}\bar i_{p,1}\\
= &-\sum_{p\mid d_2}\bar i_{p,1} \sum_{p\mid d_2}\bar i_{p,2}-\sum_{p\mid d_2}\bar i_{p,1}\sum_{p\mid d_3}\bar i_{p,2}+\sum_{p\mid d_2}\bar i_{p,1}\sum_{p\mid d_3}\bar i_{p,2}\\
&-\sum_{p\mid d_2}\bar i_{p,2} \sum_{p\mid d_2}\bar i_{p,1}-\sum_{p\mid d_2}\bar i_{p,2}\sum_{p\mid d_3}\bar i_{p,1}+\sum_{p\mid d_2}\bar i_{p,2}\sum_{p\mid d_3}\bar i_{p,1}\\
= &0.
\endaligned$$

(ii) Suppose $d_1d_3<0$.

(a) Suppose $K'=\Q(\sqrt{d_1d_2})$.
  Let $$g_1=(\sigma_p^{i_{p,1}})_{p\in S},g_2=(\sigma_p^{i_{p,2}})_{p\in S}\in \Gal(\Q^{ab}/K')\subset G^{ab}.$$
Since $g_1, g_2$ fix $K'$, we have $$\sum_{p\mid d_1 d_2}\bar i_{p,j}=0 \text{ for }j=1,2.$$
Similar as above one has $$\aligned\beta_{g_1,g_2}'=&\beta_{g_1,g_2}=\sum_{(p,q)\in R}(\bar i_{p,1}\bar i_{q,2}+\bar i_{p,2}\bar i_{q,1})
+\sum_{p\mid d_1d_2}(\bar i_{-1,1}\bar i_{p,2}+\bar i_{-1,2}\bar i_{p,1})\\
= & 0+\bar i_{-1,1}\sum_{p\mid d_1d_2}\bar i_{p,2}+\bar i_{-1,2}\sum_{p\mid d_1d_2}\bar i_{p,1}\\
= & 0.
\endaligned$$

(b) Suppose $K'=\Q(\sqrt{d_1d_3})$ (similar proof for  $K'=\Q(\sqrt{d_2d_3})$). Let $$g_1=(\sigma_p^{i_{p,1}})_{p\in S},g_2=(\sigma_p^{i_{p,2}})_{p\in S}\in \Gal(\Q^{ab}/K')\subset G^{ab}.$$
Since $g_1, g_2$ fix $K'$ and $d_1d_3<0$, we have $$\bar i_{-1,j}+\sum_{p\mid d_1 d_3}\bar i_{p,j}= 0 \text{ for }j=1,2.$$
Therefore $$\aligned\beta_{g_1,g_2}'=&\beta_{g_1,g_2}=\sum_{(p,q)\in R}(\bar i_{p,1}\bar i_{q,2}+\bar i_{p,2}\bar i_{q,1})+\sum_{p\mid d_1d_2}(\bar i_{-1,1}\bar i_{p,2}+\bar i_{-1,2}\bar i_{p,1})\\
=&\sum_{p\mid d_1}\bar i_{p,1} \sum_{p\mid d_2}\bar i_{p,2}+\sum_{p\mid d_1}\bar i_{p,1}\sum_{p\mid d_3}\bar i_{p,2}+\sum_{p\mid d_2}\bar i_{p,1}\sum_{p\mid d_3}\bar i_{p,2}\\
&+\sum_{p\mid d_1}\bar i_{p,2} \sum_{p\mid d_2}\bar i_{p,1}+\sum_{p\mid d_1}\bar i_{p,2}\sum_{p\mid d_3}\bar i_{p,1}+\sum_{p\mid d_2}\bar i_{p,2}\sum_{p\mid d_3}\bar i_{p,1}\\
&+\sum_{p\mid d_1d_2}(\bar i_{-1,1}\bar i_{p,2}+\bar i_{-1,2}\bar i_{p,1})\\
= &-(\bar i_{-1,1}+\sum_{p\mid d_3}\bar i_{p,1} )\sum_{p\mid d_2}\bar i_{p,2}-(\bar i_{-1,1}+\sum_{p\mid d_3}\bar i_{p,1})\sum_{p\mid d_3}\bar i_{p,2}+\sum_{p\mid d_2}\bar i_{p,1}\sum_{p\mid d_3}\bar i_{p,2}\\
&-(\bar i_{-1,2}+\sum_{p\mid d_3}\bar i_{p,2}) \sum_{p\mid d_2}\bar i_{p,1}-(\bar i_{-1,2}+\sum_{p\mid d_3}\bar i_{p,2})\sum_{p\mid d_3}\bar i_{p,1}+\sum_{p\mid d_2}\bar i_{p,2}\sum_{p\mid d_3}\bar i_{p,1}\\
&+\sum_{p\mid d_1d_2}(\bar i_{-1,1}\bar i_{p,2}+\bar i_{-1,2}\bar i_{p,1})\\
= &-\bar i_{-1,1}\sum_{p\mid d_2d_3}\bar i_{p,2}- \bar i_{-1,2}\sum_{p\mid d_2d_3}\bar i_{p,1}+\sum_{p\mid d_1d_2}(\bar i_{-1,1}\bar i_{p,2}+\bar i_{-1,2}\bar i_{p,1})\\
= &\bar i_{-1,1}\sum_{p\mid d_1d_3}\bar i_{p,2}+ \bar i_{-1,2}\sum_{p\mid d_1d_3}\bar i_{p,1}\\
= &-\bar i_{-1,1}\bar i_{-1,2}- \bar i_{-1,2}\bar i_{-1,1}= 0.
\endaligned$$
Therefore $L/K$ satisfies the condition $(*)$ over $\Q$, hence $\Cor_{K/\Q}(\Xi, \chi) \in \Br(X^c)$.

Since $\Br(X^c)/\Br_0(X^c)\cong H^3(\Gal(K/\Q),\Z)\cong \Z/2\Z$, we only need to show $\Cor_{K/\Q}(\Xi, \chi)$ is nontrivial. Recall $$F=\Q(\xi_N), K\subset F\subset \Q^{ab} \text{ and } L=F(\sqrt{\Delta}).$$ By  Proposition \ref{form}, we only need to show $\chi$ is not the restriction of a character of $\Gal(\bar \Q/\Q)$. We assume that $\chi$  is the restriction of a character $\Gal(\bar\Q/\Q)$, hence $\chi$ is trivial on the closure of the commutator subgroup of $\Gal(\bar \Q/\Q)$  which is just $\Gal(\bar\Q/\Q^{ab})$. Then $\chi$ is trivial on $\Gal(\bar\Q/\Q^{ab})$. On the other hand, $\chi$ factors through $\Gal(L/K)$. Therefore we have $\chi$ is trivial on $\Gal(\bar\Q/L\cap \Q^{ab})$. Note that $L/\Q$ is non-abelian (see Lemma \ref{gal}), hence $F=L\cap \Q^{ab}$. Therefore $\chi$ is trivial on $\Gal(\bar\Q/F)$, which is impossible since $\chi$ factors through $\Gal(L/K)$ and is nontrivial on $\Gal(L/F)$.
\end{proof}

For the cyclotomic case, the $2$-torsion part of $\Br(X^c)/\Br_0(X^c)$ can also be given by Theorem \ref{br}.
\begin{proposition}\label{cyc} Let $m\not \equiv 2 \mod 4$. Let $K=\Q(\xi_m)$ be a cyclotomic field and $L=K(\{\sqrt{u_{pq}}\}_{p<q\in S_m})$. Let $X$ be the affine variety over $\Q$ defined by $N_{K/\Q}(\Xi)=\alpha\in \Q^\times$. Then the 2-torsion subgroup of $\Br(X^c)/\Br_0(X^c)$ is generated by all $\Cor_{K/\Q}(\Xi, \chi),$ where $\chi$ runs through all characters in $\Hom(\Gal(L/K),\Q/\Z)\subset \Hom(\Gal(\bar \Q/K),\Q/\Z)$.
\end{proposition}
\begin{proof} Let $\chi$ be such a nontrivial character of $\Gal(\bar \Q/K)$. Then there is a subfield $L'\supset K$ of $L$ with $[L':K]=2$ such that $\chi$ factors through $\Gal(L'/K)$. And $L'/\Q$ is Galois and non-abelian by \cite[Theorem 1 and Proposition 1]{YZ}. That $L'/K$ satisfies the condition $(*)$ follows from the fact:

{\it suppose the group $M$ is cyclic and there is the following central extension $$0\rightarrow \Z/2\Z\ \rightarrow G \rightarrow M \rightarrow 0,$$

then $G$ is abelian.}\\
Indeed, let $M=<\beta>$ and $\Z/2\Z=<\gamma>$. Let $\hat \beta$ be a fixed lifting of $\beta$ in $G$, then we can see any element in $G$ has the form $\hat \beta^i \gamma^j$ with $i,j\in \Z$. Note that $\gamma$ is in the centre of $G$, obviously $G$ is abelian.
Therefore, $\Cor_{K/\Q}(\Xi, \chi)\in \Br(X^c)$ and is nontrivial by Theorem \ref{br}.

Let $d=\# S_m$. On the other hand, the Galois group $\Gal(L/K)\cong (\Z/2\Z)^{d(d-1)/2}$ by the linear independent $u_{pq}$ in $K^\times/{K^\times}^2$ (see \cite[Lemma 4]{YZ}). Therefore the subgroup of $\Br(X^c)/\Br_0(X^c)$ generated by all $\Cor_{K/\Q}(\Xi, \chi)$ is of 2-rank $d(d-1)/2$. Using the K\"{u}nneth formula (\cite[Chapter II, Section 1, Exercise 7]{NSW}), we can calculate that the 2-rank of $H^3(\Gal(K/k),\Z)$  is also $d(d-1)/2$.
Since $\Br(X^c)/\Br_0(X^c)\cong H^3(\Gal(K/k),\Z)$, the 2-rank of $\Br(X^c)/\Br_0(X^c)$  is $d(d-1)/2$.
Therefore all $\Cor_{K/\Q}(\Xi, \chi)$ generate the 2-torsion subgroup of $\Br(X^c)/\Br_0(X^c)$.
\end{proof}

Finally we will give an explicit
example associated to a cyclotomic field.
\begin{example} \label{exa:cyclotomic}
Let $K=\Q(\xi_{7\cdot 53})$. For any positive rational number $\alpha$, we can write $\alpha=2^{s_1}7^{s_2}53^{s_3} p_1^{e_1}\cdots p_g^{e_g}$, where $p_1,\cdots, p_g$ are distinct primes which do not divide $742$,  and $s_1, \cdots, s_3$ are integers, and $e_1, \cdots, e_g$ are nonzero integers.
Let $D(\alpha)=\{p_1, \cdots, p_g \}$.
Denote
$$\aligned
& D_1=\{p\in D(\alpha): \left(\frac{-7}{p}\right)=\left(\frac{53}{p}\right)=-1 \}\cr
& D_2=\{p\in D(\alpha): \left(\frac{-7}{p}\right)=\left(\frac{53}{p}\right)=1 \text{ and } \left(\frac{5+2\sqrt{-7}}{p}\right)=-1\}.
\endaligned $$
Then the equation
$N_{K/\Q}(\Xi)=\alpha$ is solvable over $\Q$ if and only if the following conditions hold:
\begin{enumerate}[(i)]
\item The equation $N_{K/\Q}(\Xi)=\alpha$ is solvable over $\Q_p$ for each place $p$.
\item (BM-obstruction): $(-1)^{\sum_{p_i\in D_1}e_i/2+\sum_{p_i\in D_2}e_i} = (-1)^{s_2}\cdot \left(\frac{\alpha,-1}{2}\right)$.
\end{enumerate}
\end{example}
\begin{proof}
Let $X$ be the affine variety defined by $N_{K/\Q}(\Xi)=\alpha\in \Q^\times$. We can see $$\Br(X^c)/\Br_0(X^c)\cong H^3(\Gal(K/\Q),\Z)\cong \Z/2\Z.$$

It is easy to verify that $K(\sqrt{5+2\sqrt{-7}})/\Q$ is Galois. Let $L=K(\sqrt{5+2\sqrt{-7}})$.
Let $\chi$ be the unique nontrivial character of $\Gal(\bar \Q/K)$ which factors through $\Gal(L/K)$.
Then $\Cor_{K/\Q}(\Xi,\chi)$ is the unique generator of $\Br(X^c)/\Br_0(X^c)$ by Theorem \ref{br}. The proof follows from a similar argument as in Example \ref{exa:biquadratic-1}.
\end{proof}

\bf{Acknowledgment} \it{The author would like to thank Professors Colliot-Th\'el\`ene and Derenthal for helpful discussions and comments. The work was supported by National Natural Science Foundation of China (Grant No. 11371210), the National Key Basic Research Programm (Grant No. 2013CB834202) and grant DE 1646/2-1 of the Deutsche Forschungsgemeinschaft. The author is grateful to the referee for careful reading and very helpful
comments.}


\end{document}